\documentclass{amsart}
\usepackage{graphicx} 
\usepackage{amsfonts,amsmath,amssymb,amsthm}
\usepackage{mathtools}
\usepackage{hyperref}

\usepackage[%
    backend = biber,
    style = numeric,
    citestyle = numeric,
    maxalphanames = 99,
    maxnames = 99,
    backref = true,
    isbn = false,
    doi = false,
    url = false,
    ]{biblatex} 
\newbibmacro{string+url}[1]{
    \iffieldundef{url}{
          #1
    }{
      \href{\thefield{url}}{#1}
    }
}

\DeclareFieldFormat
    [article,misc,inbook,incollection,inproceedings,patent,thesis,unpublished]
  {title}{\usebibmacro{string+url}{\mkbibquote{#1\isdot}}}

\theoremstyle{plain}
\newtheorem{thm}{Theorem}[section]

\newtheorem{cnj}[thm]{Conjecture}
\newtheorem{crl}[thm]{Corollary}

\newtheorem{lm}[thm]{Lemma}

\newtheorem{prp}[thm]{Proposition}
    
\newcommand{\bbN}{\mathbb{N}}

\newcommand{\oGL}{\operatorname{GL}}
\newcommand{\oO}{\operatorname{O}}
\newcommand{\oSp}{\operatorname{Sp}}
\newcommand{\oSpec}{\operatorname{Spec}}
\newcommand{\oSym}{\operatorname{Sym}}

\newcommand{\bV}{\mathbf{V}}

\title{Topological Noetherianity for powers of algebraic representations}
\author{Alessandro Danelon}
\thanks{Department of Mathematics, University of Michigan, 530 Church St, Ann Arbor, Michigan 48109, email: \href{adanelon@umich.edu}{adanelon@umich.edu}, ORCID:0000-0003-4574-9552}
\date{April 6, 2026}

\begin{document}

\begin{abstract}
    Powers of a polynomial $\oGL$-representation are topologically Noetherian under the action of $\oSym \times \oGL$.
    We extend this result to powers of algebraic representations of the orthogonal and the symplectic groups, proving topological Noetherianity under the action of $\oSym \times \oO$ and $\oSym \times \oSp$ respectively.
    This work builds on \cite{chiu-danelon-draisma-eggermont-farooq} and \cite{eggermont-snowden}, and it provides further evidence that infinite powers of topologically Noetherian varieties remain topologically Noetherian up to permutations.
\end{abstract}

\maketitle

\section{Introduction}
Let $G$ be a group acting on a topological space $X$.
We say that $X$ is \textit{topologically $G$-Noetherian} if every descending chain of closed $G$-stable subsets stabilizes.
We consider algebraic varieties of tensors that are stable under the simultaneous and independent actions of the infinite symmetric group and the infinite orthogonal (or symplectic) group, and we show that they are topologically Noetherian up to the action of these groups.

A long-standing question in this area of mathematics concerns the boundary of Noetherianity for tensor varieties, see e.g. \cite{draisma, eggermont-snowden, nagpal-sam-snowden, bik-danelon-draisma:topologicalN, chiu-danelon-draisma-eggermont-farooq, ganapathy:nonnoethtwo}, and this work claims a novel class of examples in the Noetherian side.
This paper strengthens the idea (first introduced in \cite{chiu-danelon-draisma-eggermont-farooq}) that topological Noetherianity holds for infinite powers of topologically Noetherian varieties up to permutations.
A natural follow-up step in this philosophy is to show that infinite powers of half-spin varieties (defined in \cite{chiu-draisma-eggermont-seynnaeve-tairi}) are topologically Noetherian up to the action of the half-spin group and the infinite symmetric group.
We plan to work on this latter conjecture in the future.

\subsection{Main result}
Let $k$ be an algebraically closed field of characteristic zero, and define the infinite general linear group $\oGL$ as the union of the general linear groups $\oGL_n(k)$ on $k^n$.
We can think of the elements of $\oGL$ as $\mathbb{N}\times \mathbb{N}$-matrices with a $g\in \oGL_n(k)$ in the top left corner and an infinite identity matrix in the bottom right corner.
The infinite symmetric group $\oSym$ is the group of permutations of $\mathbb{N}$.
Let $\oO = \bigcup_{n\geq 1}\oO_n(k)$ be the infinite orthogonal group and $\oSp = \bigcup_{n\geq 1}\oSp_{2n}(k)$ be the infinite symplectic group.

A \textit{polynomial} representation of $\oGL$ is a subquotient of a finite direct sum of tensor powers of the standard representation $\bV = \bigcup_{n \geq 0}k^n$.
Let $\bV_* = \bigcup_{n \geq 0} (k^n)^*$ be the \textit{restricted} dual of $\bV$.
The group $\oGL$ acts on $\bV_*$ in the induced manner.
An \textit{algebraic} representation of $\oGL$ is a subquotient of a finite direct sum of tensor powers of $\bV^{\otimes n} \otimes \bV^{\otimes m}_*$ with $m,n \in \mathbb{Z}_{\geq 0}$.
An \textit{algebraic} representation of $\oO$ and $\oSp$ is a subquotient of a finite direct sum of tensor powers of the standard representation $\bV$.
As $\bV \cong \bV_*$ as $\oO$ and $\oSp$-representations, quotients of finite direct sums of $\bV^{\otimes n} \otimes \bV^{\otimes m}_*$ belong to the category of $\oO$ and $\oSp$ algebraic representations.

Let $E$ be an algebraic $G$-representation where $G$ is one of the above groups.
Its dual $E^*$ coincides with the closed points of $\operatorname{Spec}(k[E])$ from which it inherits the Zariski topology.

Consider a topological space $X$ with an action of a group $G$ and restrict the attention to closed subsets stable under the action of $G$.
The space $X$ is \textit{topologically} $G$-Noetherian if every descending chain of closed $G$-stable subsets stabilizes.
If $X$ is an algebraic variety, topological $G$-Noetherianity is equivalent to the fact that $X$ is cut out set-theoretically by the $G$-orbit of finitely many equations.

The main result of this paper is:

\begin{thm}\label{thm:main}
    Let $G$ be one of the groups $\oGL, \oO,$ or $\oSp$.
    Let $E$ be an algebraic $G$-representation.
    Consider the action of the group $\oSym \times G$ on the countable direct sum $E^{\oplus  \mathbb{N}}$ with $\oSym$ permuting the copies of $E$ and $G$ acting simultaneously on each copy of $E$.
    Then the dual space $(E^{\oplus \mathbb{N}})^*$ is topologically $\oSym \times G$-Noetherian.
\end{thm}

Actually, a more general version of this theorem is true:

\begin{crl}\label{crl:extension}
    Let $E_1, \dots, E_k$ be $G$-representations.
    The group $\oSym^k \times G$ acts on $E_1^{\oplus \mathbb{N}} \oplus \cdots \oplus E_k^{\oplus \mathbb{N}}$ with $i$-th copy of $\oSym$ permuting the summands of $E_i^{\oplus \mathbb{N}}$, and $G$ acting simultaneously on each $E_j$.
    Then the dual space of $E_1^{\oplus \mathbb{N}} \oplus \cdots \oplus E_k^{\oplus \mathbb{N}}$ is topologically $\oSym^k \times G$-Noetherian.
\end{crl}

\subsection{Examples}
For a vector space $E$, the countable product of its dual $(E^*)^\mathbb{N}$ coincides with the dual of its countable direct sum $(E^{\oplus \mathbb{N}})^*$.
We write a point $p \in (E^*)^{\oplus \mathbb{N}}$ as $p = (p_i)_{i \in \mathbb{N}}$ where $p_i\in E^*$ is the projection of $p$ on its $i$-th component.

(a) Consider the standard representation $\bV$ of $\oGL$.
Its dual $\bV^*$ can be identified with infinite column vectors with possibly infinitely many nonzero entries; these are the $k$-points of the spectrum of $\oSym(\bV)$.
The product $(\bV^*)^\mathbb{N}$ of countably many copies of $\bV^*$ carries commuting actions of the infinite symmetric group, permuting the factors, and $\oGL$, acting simultaneously on each factor.
We can think of $(\bV^*)^\mathbb{N}$ as $\mathbb{N}~\times~\mathbb{N}$-matrices where $\oSym$ permutes the columns and $\oGL$ acts by linear combinations of rows.
The coordinate ring of this space is $\oSym(\bV^{\oplus \mathbb{N}})$.

(b) Consider the restricted dual $\bV_*$ as an algebraic representation of $\oGL$.
We represent its elements as row vectors.
We can then think of $((\bV_*)^*)^\mathbb{N}$ as $\mathbb{N}~\times~\mathbb{N}$-matrices with possibly infinitely many nonzero entries where $\oSym$ permutes the rows and $\oGL$ acts by linear combinations of columns (multiplication now happens on the right).

(c) Consider $M = \bV \otimes \bV$.
Elements of $M$ are infinite-by-infinite matrices with finitely many nonzero entries.
The dual $M^*$ is made of matrices with possibly infinitely many nonzero entries.
Examples of closed $\oGL$-stable subsets of $M^*$ are the variety of symmetric matrices and the variety of antisymmetric matrices.
A point in the space $((\bV \otimes \bV)^*)^\mathbb{N}$ is then a countable collection of $\mathbb{N}\times \mathbb{N}$-matrices.
Its coordinate ring is $\oSym(\bigoplus_{i\in \mathbb{N}}\bV \otimes \bV)$.
A closed subvariety of $((\bV \otimes \bV)^*)^\mathbb{N}$ is given by the points whose entries are either a symmetric matrix or an antisymmetric matrix (\cite[Example~1.6]{chiu-danelon-draisma-eggermont-farooq}).

(d) Consider the algebraic $\oGL$-representation $A = \bV \otimes \bV_*$, and consider its dual $A^*$.
Geometrically this spaces is the same as above except that $\oGL$ acts as $(g^{-1})^t$ on the second factor.
In any case, Theorem~\ref{thm:main} holds also for $(A^*)^\mathbb{N}$.

(e) Restricting the action of $\oGL$ to either $\oO$ or $\oSp$ on the above examples gives new examples of Theorem~\ref{thm:main} for these groups.

(f) For $1\leq i < j \leq k$ consider map $\phi_{i,j}: \bV^{\otimes k} \to \bV^{\otimes (k-2)}$ sending $v_1 \otimes \cdots \otimes v_k \to (v_i^t v_j) v_1 \otimes \cdots \otimes \hat{v}_i \otimes \cdots \otimes \hat{v}_j \otimes \cdots \otimes v_k$ and extended linearly.
Let $\bV^{[k]}$ be the intersection of all the kernels of $\phi_{i,j}$.
The space $\bV^{[k]}$ is an algebraic representation of $\oO$.
Theorem~\ref{thm:main} guarantees that any subvariety in $((\bV^{[k]})^{\oplus \mathbb{N}})^*$ is topologically $\oSym \times \oO$-Noetherian.

\subsection{Connection to the literature}

Over infinite fields, Draisma showed in \cite{draisma} that the dual space $E^*$ of a polynomial $\oGL$-representation $E$ is topologically $\oGL$-Noetherian.
Later, over algebraically closed fields, Eggermont--Snowden in \cite{eggermont-snowden} extended Draisma's result to algebraic representations of the infinite orthogonal group $\oO$ and of the infinite symplectic group $\oSp$.

\begin{sloppypar}
In characteristic zero, the work \cite{chiu-danelon-draisma-eggermont-farooq} extends topological Noetherianity up to symmetry to powers of polynomial $\oGL$-representations, i.e. $\oSpec(\oSym(E^\bbN))$, with the action of $\oSym\times \oGL$.
This is a consequence of the following more general theorem.
\end{sloppypar}

\begin{thm}[\cite{chiu-danelon-draisma-eggermont-farooq}] \label{thm:ours}
    Consider polynomial $\oGL$-representations $E_1,$ $\dots,$ $ E_k$.
    Consider the action of the group $\oSym^k \times \oGL$ on $E_1^\mathbb{N} \times E_2^\mathbb{N}\times \cdots E_k^\mathbb{N}$ where the $i$-th copy of $\oSym$ permutes the copies of $E_i$ and $\oGL$ acts simultaneously on each factor $E_j$.
    Then the space $(E_1^\mathbb{N})^*\times \cdots \times (E_k^\mathbb{N})^*$ is topologically $\oSym^k \times \oGL$-Noetherian.
\end{thm}

Our work is a natural follow up to \cite{chiu-danelon-draisma-eggermont-farooq} and to \cite{eggermont-snowden}.
The proof strategy is an extension of \cite{eggermont-snowden}.
However, our result doesn't directly follow from the result of \cite{eggermont-snowden} as the varieties we are dealing with are a generalization of the varieties involved in \cite{eggermont-snowden}.
Specifically, \cite[Proposition~6]{eggermont-snowden} is applied to a single polynomial $G$-representation $E$, while in our case we are dealing with the infinitely many copies $E^{\oplus \mathbb{N}}$.
So \cite[Proposition~6]{eggermont-snowden} is not directly available.
The new input is the observation that one only needs to work with $(E_n)^{\oplus k}$ seen as a polynomial $G$-representation, without the action of $\oSym$.

\subsection{Future directions}
It would be interesting to know if topological Noetherianity can be extended to the simultaneous actions of $\oGL, \oO, \oSp$.
In other words, if $E_1, E_2, E_3$ are, respectively, algebraic representations of $\oGL, \oO, \oSp$, is it true that the dual of $E_1 \oplus E_2 \oplus E_3$ is topologically Noetherian up to $\oGL\times\oO\times\oSp$ acting on the corresponding term?
If the above holds (as we believe), an even more general version of Theorem~\ref{thm:main} might be true:

\begin{cnj}
    Let $I, J, L$ be disjoint index sets.
    Let $E_i$ be an algebraic $\oGL$-representation if $i \in I$, an algebraic $\oO$-representation if $i \in J$, an algebraic $\oSp$-representation if $i \in L$.
    The dual of 
    \[
    \bigoplus_{i\in I} E_i^{\oplus \mathbb{N}} \oplus  \bigoplus_{j\in J} E_j^{\oplus \mathbb{N}} \oplus \bigoplus_{l\in L} E_l^{\oplus \mathbb{N}}
    \]
    is topologically $\oSym^{|I|}  \times \oGL \times \oSym^{|J|} \times \oO \times \oSym^{|L|} \times  \oSp $-Noetherian, where $\oGL$ acts simultaneously on the copies $E_i^*$ with $i \in I$, $\oO$ on the copies $E_j^*$ with $j \in J$, and $\oSp$ on the copies $E_l^*$ with $l \in L$.
\end{cnj}

The authors of \cite{chiu-draisma-eggermont-seynnaeve-tairi} introduced the half-spin varieties and showed that these are topologically Noetherian up to the action of $\operatorname{Spin}$, the direct limit of all spin groups.
A natural and interesting follow-up question is if powers of half-spin varieties are topologically Noetherian up to $\oSym \times \operatorname{Spin}$.
We plan to think about these questions in the future.

\subsection{The proof strategy}
The proof proceeds in two steps.
We first show that if Theorem~\ref{thm:main} holds for any of the groups $\oGL, \oO$, or $\oSp$ then it holds for the other two.
After this we show that Theorem~\ref{thm:main} holds for the group $\oGL$.
Below we extend to our setting some results of \cite{eggermont-snowden}.

The following lemma is the extension of \cite[Lemma~2.1]{eggermont-snowden} to include the symmetry of $\oSym$.
The key difference is the assumption of topological $\oSym\times H$-Noetherianity. See below for details.

\begin{lm}\label{lm:morphismtrick}
    Let $G$ and $H$ be distinct groups belonging to $\{\oGL, \oO, \oSp \}$.
    Consider a group homomorphism $\phi: H \to G$ such that the standard representation of $G$ pulls back to an algebraic representation of $H$.
    Assume that for an algebraic $H$-representation $E$ the space $(E^{\oplus \mathbb{N}})^*$ is topologically $\oSym \times H$-Noetherian.
    Then for any algebraic representation $F$ of $G$ the space $(F^\mathbb{N})^*$ is topologically $\oSym \times G$-Noetherian.
\end{lm}

\begin{proof}
    The assumptions guarantee that an algebraic representation $F$ of $G$ pulls back along $\phi$ to an algebraic representation $F$ of $H$, and that the space $(F^{\mathbb{N}})^*$ is topologically $\oSym \times H$-Noetherian.
    Consider a descending chain of closed $\oSym\times G$-stable subsets of $(F^\mathbb{N})^*$.
    These are also stable under $\oSym\times H$ and therefore the descending chain stabilizes.
\end{proof}

\begin{prp}
    If Theorem~\ref{thm:main} holds for one group in $\{\oGL, \oO, \oSp \}$, then it holds for the other two.
\end{prp}

\begin{proof}
For $G, H \in \{\oGL, \oO, \oSp \}$ we need to find a group homomorphism $\phi: H \to G$ such that we can apply Lemma~\ref{lm:morphismtrick}.
As there is no action of $\oSym$ involved, we can verbatim apply the proof of \cite[Proposition~4]{eggermont-snowden}.

    \end{proof}

\subsection{Auxiliary notation and result}\label{ssec:notandres}

A \textit{polynomial} representation of $\oGL \times \oGL$ is a subquotient of a finite direct sum of $\bV^{\otimes n}\otimes \bV^{\otimes m}$ with $n,m \in \mathbb{Z}_{\geq 0}$.
The use of \cite[Theorem~1.1]{chiu-danelon-draisma-eggermont-farooq} gives the following extension to \cite[Proposition~2.3]{eggermont-snowden}.

\begin{prp}
    Let $E$ be a polynomial representation of $\oGL \times \oGL$.
    Then $(E^{\oplus \mathbb{N}})^*$ is topologically $\oSym \times \left ( \oGL \times \oGL \right )$-Noetherian.
\end{prp}

\begin{proof}
    Since $\oGL$ embeds diagonally into $\oGL \times \oGL$, the vector space $E$ is a polynomial $\oGL$-representation.
    By Theorem~\ref{thm:ours} with $k = 1$, the space $(E^{\oplus \mathbb{N}})^*$ is topologically $\oSym \times \oGL$-Noethe\-rian, and hence it is topologically $\oSym \times \left ( \oGL \times \oGL \right )$-Noetherian.
 \end{proof}

\section{The general linear case}
We extend \cite[Section~4]{eggermont-snowden} to our setting.
Let $H$ be the subgroup of $G = \oGL \times \oGL$ given by elements of the form $(g, {}^t g^{-1})$.
An algebraic $\oGL$-representation is then given by the restriction of a polynomial $\oGL \times \oGL$-representation $E$ to the subgroup $H$.
Let $E$ denote a fixed $\oGL \times \oGL$-representation.
We prove that $(E^{\oplus \mathbb{N}})^*$ is topologically $\oSym \times H$-Noethe\-rian.

Let the group $G$ act on an element $A \in M = \bV \otimes \bV$ by $(g,h) \cdot A = g\; A {\;}^t h$.
Its induced action on an element $A \in M^*$ is $(g,h)\cdot A = {\;}^t g^{-1} A h^{-1}$.
Note that the stabilizer in $G$ of the infinite identity matrix $I\in M^*$ is $H$.

Since $E$ is a polynomial $\oGL\times\oGL$-representation it is a subquotient of $\bigoplus_{i = 1}^l\bV^{\otimes m_i} \otimes \bV^{\otimes n_i}$ for some $l$.
Define $E_n$ to be the restriction of $E$ to the vectors involving only the first $n$ coordinates.
In other words, $E_n$ is a subquotient of $\bigoplus_{i = 1}^l (k^n)^{\otimes m_i} \otimes (k^n)^{\otimes n_i}$.
We call $E_n$ the \textit{level} $n$ of $E$.

For $M = \bV \otimes \bV$ the space $M_n$ is $k^n \otimes k^n$.
In $M^*$, consider the subspace $U^*$ of upper-triangular matrices, and let $U^*_1$ be the subspace of $U^*$ where the entries on the diagonal are all $1$.
Define $W^*$ as $ U^* \times U^*_1$ and consider the map $\phi: W^* \to M^*$ sending $(u,v)$ to ${}^t u v$.
Let $\pi_n: M^* \to M_n$ be the projection on the top-left $n \times n$-block.
Let $\phi_n$ be the restriction of $\phi$ to $W_n$ mapping in $M_n$.

\begin{lm}\label{lm:Htoh}
    Let $X$ be an affine variety over $k$ and consider a regular map $h: W^*_n \times X \to k$.
    Then there exists a regular map $H : M^*_n \times X \to k$ such that for any $((u,v),x)\in W^*_n \times X$ we have $H(\phi_n(u,v), x) = m(u)h((u,v),x)$ where $m(u)$ is a monomial depending only on the diagonal entries of $u$.
\end{lm}

\begin{proof}
    \cite[Lemma~8]{eggermont-snowden}.
\end{proof}

For a polynomial $\oGL$-representation $E$ the action of $\oGL_n$ extends to an action of its Lie algebra $M_n$, and therefore to an action of $M$ on $E$ and on $E^*$.
The action of $U$ extends to an action of $U^*$ on $E$.
Under this action, for a $u\in U^*$ and an $e \in E$ the element $u\cdot e$ can be computed by the restriction of $u$ to the top-left $n\times n$-block where $n$ is the level $E_n$ where the element $e$ sits.
The same reasoning works for the action of $U^*$ on $(E_n)^*$.

Consider a $\oSym \times H$-stable closed subset $Z$ of $(E^{\oplus \mathbb{N}})^*$, and let $Z^+$ be the closure of the $\oSym \times G$-orbit of $\{I\} \times Z$ in $M^* \times (E^{\oplus \mathbb{N}})^*$ with $\oSym$ acting trivially on the first component.
Note that $M^* \times (E^{\oplus \mathbb{N}})^*$ is isomorphic to the closed subvariety in $((M\oplus E)^\mathbb{N})^*$ of points whose entries in  $M^*$ are all equal.
In particular $Z^+$ is a closed subvariety of $((M\oplus E)^\mathbb{N})^*$.

\begin{prp}\label{prp:ZZplus}
    The intersection $(\{I\}\times (E^{\oplus \mathbb{N}})^*) \cap Z^+$ is equal to $\{I\} \times Z$.
\end{prp}

    Proposition~\ref{prp:ZZplus} is the generalization of \cite[Proposition~6]{eggermont-snowden} to the group $\oSym \times H$.
    Indeed, we cannot directly apply \cite[Proposition~6]{eggermont-snowden} as $(E^{\oplus \mathbb{N}})$ is not a polynomial representation of $G$ (it is not finite).
    However, a close look at the proof of \cite[Proposition~6]{eggermont-snowden} shows that we only need to look at $(E_n^{\oplus k})$, that is a polynomial $G$-representation.
    This observation also allows to extend the proof of Theorem~\ref{thm:main} to prove Corollary~\ref{crl:extension}.
    Indeed, we can apply the same argument and substitute $(E_n^{\oplus k})$ with $(E_1)_n^{\oplus k} \oplus \cdots \oplus (E_l)_n^{\oplus k}$ for some polynomial $G$-representations $E_i$.

An immediate consequence of Proposition~\ref{prp:ZZplus} is topological $\oSym\times \oGL$-Noetheria\-nity of $(E^{\oplus \mathbb{N}})^*$:

\begin{crl}
    The map $Z \to Z^+$ is an injection from the set of closed $\oSym \times H$-stable subsets of $(E^{\oplus \mathbb{N}})^*$ to 
    the set of closed $\oSym \times G$-stable subsets of $M^* \times (E^{\oplus \mathbb{N}})^*$.
    In particular, the space $(E^{\oplus \mathbb{N}})^*$ is topologically $\oSym \times H$-Noetherian.
\end{crl}

The above corollary plays the role of \cite[Corollary~7]{eggermont-snowden}.
The main difference is the use of the fact that $((M\oplus E)^\mathbb{N})^*$ is topologically $\oSym \times \oGL$-Noetherian.

\begin{proof}
    Consider a chain $Z^\bullet$ of $\oSym \times H$-stable closed subsets of $(E^{\oplus \mathbb{N}})^*$ and let $(Z^\bullet)^+$ be the corresponding chain in $M^* \times (E^{\oplus \mathbb{N}})^*$.
    The space $((M\oplus E)^\mathbb{N})^*$ is $\oSym \times \oGL$-Noetherian by Theorem~\ref{thm:ours}, and hence $\oSym \times G$-Noetherian.
    As $M^* \times (E^{\oplus \mathbb{N}})^*$ is a closed subvariety of $((M\oplus E)^\mathbb{N})^*$, the $\oSym \times G$-stable chain $(Z^\bullet)^+$ stabilizes, and Proposition~\ref{prp:ZZplus} guarantees that so does $Z^\bullet$.
    As every chain of $\oSym \times H$-stable closed subsets stabilizes, the space $(E^{\oplus \mathbb{N}})^*$ is topologically $\oSym \times H$-Noetherian.
\end{proof}

For completeness, we write the proof of Proposition~\ref{prp:ZZplus} showing that \cite[Section~4]{eggermont-snowden} follows through in our case.
The proof splits into two steps.
Let $Z$ be a $\oSym \times H$-stable closed subset of $(E^{\oplus \mathbb{N}})^*$.

\begin{enumerate}
    \item Given any defining equation $f$ for $Z$, define a regular function $H_f$ on $M^*\times (E^{\oplus \mathbb{N}})^*$ vanishing on $Z^+$.
    \item Show that if a point $p$ is not a zero of a defining equation $f$ of $Z$, then it is not a zero of $H_f$.
\end{enumerate}

\begin{proof}[Proof~of~Proposition~\ref{prp:ZZplus}] 
Let $Z$ be as above.
Consider a defining equation $f$ of $Z\subset (E^{\oplus \mathbb{N}})^*$.
Since $f$ is a polynomial, the coordinates of $f$ refer to a finite number of copies of $E$, and only up to a fixed level in each of these copies.
We can then assume that $f$ is a regular function on $(E_n^{\oplus k})^*$ for some suitable $k$ and $n$.
Define the function $h_f : W^*_n \times (E^{\oplus k})^*_n \to k$ sending $(w,x)$ to $f(w \cdot x)$.
The element $w$ acts on $x$ simultaneously on each factor $(E_n^{\oplus k})^*$ as described above.

Lemma~\ref{lm:Htoh} applied to $h_f$ gives a regular function $H_f$ on $M_n^* \times (E^{\oplus k})^*_n$ such that $H_f(\phi_n(u,v), x) = m_f(u)h_f((u,v),x)$ on $W_n^* \times (E^{\oplus k})^*_n$, where $m_f$ is a monomial depending only on the diagonal entries of $u$.

We claim that the polynomial $H_f$ vanishes on $Z^+$ since it vanishes on a dense subset of $Z^+$.
Let $G_m = \oGL_m \times \oGL_m$, and let $H_m$ be the subgroup given by elements $(g, {}^tg^{-1})$ with $g\in \oGL_m$.
The product $W^*_m H_m = (ug, v{\;}^tg^{-1})$ is dense in $G_m$ because they both have dimension $2m^2$.
One can check that the fibers of $W^*_m\times H_m \to M_m \times M_m$ have dimension zero since $U^*_1$ guarantees uniqueness of LU-decomposition.
This means that the orbit under $W^*_m H_m$ of $\{I\}\times Z$ is dense in the orbit under $G_m$ of $\{I\} \times Z$.

Let $z^+$ be an element of the  $W^*_m H_m$-orbit.
We can write $z^+$ as $wh(I, z) = ({}^t(ug)^{-1} I (v {\;}^tg^{-1})^{-1}, wh\cdot z) = ({}^tu^{-1} {\;}^tg^{-1} I {\;}^tg v^{-1}, wh\cdot z) = ({}^tu^{-1} I v^{-1},wh\cdot z) = (\phi_m(u^{-1}, v^{-1}), wh z)$, where $w = (u,v)$ and $h = (g, {}^tg^{-1})$ and $z \in Z$.
Since $hz$ is a point of $Z$, it must vanish on $f$, then the evaluation of $H_f$ at this point gives:
\begin{equation*}
    H_f(z^+) = H_f(\phi(w^{-1}), whz) = m(u^{-1})f(w^{-1} wh z) =  m(u^{-1})f(h z) = 0.
\end{equation*}
Hence, $H_f$ vanishes on a dense subset of $Z^+$ and therefore on $Z^+$.

Let $z$ be a point in $(E^{\oplus \mathbb{N}})^*$ that is not on $Z$ and let $f$ be a defining equation for $Z$ such that $f(z) \neq 0$.
Then $(I, z) \in Z^+$ and $H_f((I, z)) \neq 0$ because
\begin{equation*}
    H_f((I, z)) = m_f(I) h_{f}((I,I), z) = m_f(I)f(z) \neq 0.
\end{equation*}
Indeed, the diagonal entries of $I$ are ones, so $m_f(I) \neq 0$ and $f(z) \neq 0$ by hypothesis.
\end{proof}

\section{Declarations}
The author does not have conflicts of interest. Data availability: N/A.

\printbibliography
\end{document}